\documentclass[12pt]{amsart}
\usepackage{amsmath, amssymb}
\usepackage{amsthm}
\usepackage{lastpage}
\usepackage{fancyhdr}
\usepackage{graphicx}
\usepackage{bbm}
\usepackage{empheq}
\usepackage{geometry}
\usepackage[matrix, arrow, curve]{xy}
\usepackage{float, placeins}
\usepackage{epsfig, graphics, psfrag, overpic, color, colonequals}
\usepackage{hyperref}
\usepackage[capitalise]{cleveref}

\makeatletter
\@namedef{subjclassname@2010}{%
  \textup{2010} Mathematics Subject Classification}
\makeatother

\numberwithin{equation}{section}
\newcommand*{\longhookrightarrow}{\ensuremath{\lhook\joinrel\relbar\joinrel\rightarrow}}

\DeclareMathOperator{\diam}{diam}

\newcommand{\bbN}{\mathbbm{N}}

\newcommand{\mcD}{\mathcal{D}}

\newcommand{\mcL}{\mathcal{L}}

\renewcommand{\phi}{\varphi}

\numberwithin{equation}{section}
\newtheorem{thm}[equation]{Theorem}
\newtheorem{prop}[equation]{Proposition}
\newtheorem{cor}[equation]{Corollary}
\newtheorem{lemma}[equation]{Lemma}

\newtheorem*{thm*}{Theorem}

\newtheorem*{cor*}{Corollary}
\newtheorem*{prop*}{Proposition}
\newtheorem*{lemma*}{Lemma}
\newtheorem*{mconj*}{Meta-Conjecture}
\newtheorem*{conj*}{Conjecture}
\newtheorem{cor+}{Corollary}
\newtheorem*{claim}{Claim}
\newtheorem*{challenge}{Challenge}

\theoremstyle{definition}

\newtheorem{example}[equation]{Example}
\newtheorem{defi}[equation]{Definition}
\newtheorem{rem}[equation]{Remark}

\newtheorem*{bsp*}{Example}
\newtheorem*{example*}{Example}
\newtheorem*{defi*}{Definition}
\newtheorem*{rem*}{Remark}
\newtheorem*{nota*}{Notation}

 \newtheoremstyle{TheoremNum}
        {}{}              
        {\itshape}                      
        {}                              
        {\bfseries}                     
        {.}                             
        { }                             
        {\thmname{#1}\thmnote{ \bfseries #3}}
\theoremstyle{TheoremNum}
\newtheorem{propn}{Proposition}
\newtheorem{corn}{Corollary}
\newtheorem{thmn}{Theorem}

\newcommand{\RR}{{\mathbb{R}}}
\newcommand{\ZZ}{{\mathbb{Z}}}
\newcommand{\N}{{\mathbb{N}}}

\newcommand{\QQ}{{\mathbb{Q}}}

\def\Z{\ZZ}
\def\Q{\QQ}
\def\R{\RR}

\def\ol{\overline}
\def\wt{\widetilde}
\def\ll{\langle}
\def\rr{\rangle}
\DeclareMathOperator{\Int}{int}

\def\sm{\setminus}
\newcommand{\eps}{\varepsilon}
\DeclareMathOperator{\Wh}{Wh}
\DeclareMathOperator{\cl}{cl}

\begin{document}

\baselineskip=17pt

\title{Shrinking of toroidal decomposition spaces}

\author[D. Kasprowski]{Daniel Kasprowski}
\address{Westf\"{a}lische Universit\"{a}t M\"{u}nster\\
Einsteinstrasse 62 \\
48149 M\"{u}nster \\
Germany
}
\email{daniel.kasprowski@uni-muenster.de}

\author[M. Powell]{Mark Powell}
\address{
  Department of Mathematics\\
  Indiana University \\
  Bloomington, IN 47405\\
  USA
  }
\email{macp@indiana.edu}

\address{
  Max Planck Institute for Mathematics\\
  Vivatsgasse 7\\
  53111 Bonn\\
  Germany
}

\date{}

\begin{abstract}
Given a sequence of oriented links $L^1,L^2,L^3,\dots$ each of which has a distinguished, unknotted component, there is a decomposition space~$\mathcal{D}$ of~$S^3$ naturally associated to it, which is constructed as the components of the intersection of an infinite sequence of nested solid tori.  The Bing and Whitehead continua are simple, well known examples.  We give a necessary and sufficient criterion to determine whether $\mcD$ is shrinkable, generalising previous work of F.~Ancel and M.~Starbird and others. This criterion can effectively determine, in many cases, whether the quotient map $S^3 \to S^3 / \mathcal{D}$ can be approximated by homeomorphisms.
\end{abstract}



\maketitle

\section{Introduction}

In this paper $\N := \N_{>0}$, the set of positive integers, and $\N_0 := \N \cup \{0\}$.

A \emph{decomposition} of a manifold $X$ is a collection $\mathcal{D} = \{\Delta_i\}$ of pairwise disjoint closed subsets of $X$ with $\bigcup\,\mathcal{D} = X$.  A decomposition of a compact manifold is said to be \emph{shrinkable}, in the sense of R.~H.~Bing, if the associated quotient map which identifies the elements of $\mathcal{D}$ to points can be approximated by homeomorphisms, so that there is a sequence of homeomorphisms $\{f_j \colon X \to X/\mathcal{D}\}$ converging to the quotient map $f\colon X \to X/\mathcal{D}$. The adjective shrinkable arises from the method typically used to construct such a sequence.  One constructs homeomorphisms $h_j \colon X \to X$ such that the subsets $h_j(\Delta_i)$ become progressively smaller; that is, they shrink.  See \cref{S:notations} for more details.

In particular, for a shrinkable decomposition, the spaces $X$ and $X/\mathcal{D}$ are seen to be homeomorphic.  When specifying a decomposition $\mathcal{D}$, following custom we only specify the subsets which are not singletons.

Suppose we are given a sequence of oriented links $\mcL=L^1,L^2,L^3,\dots$ in $S^3$, where $L^i=L_0^i \, \sqcup\, L_1^i \, \sqcup \dots \sqcup \, L_{n_i}^i$ is an oriented $(n_i+1)$-component link with a specified component $L_0^i$ unknotted. The aim of this paper is to give general conditions to decide whether the following decomposition $\mcD$ of $S^3$ associated to $\mcL$ shrinks.  Each link $L^i$ determines a link $L^i_1 \, \sqcup \dots \sqcup \, L^i_{n_i}$ in $S^3 \setminus \nu L_0^i$, the complement of an open regular neighbourhood of $L_0^i$. The orientation of $L^i_0$ and the embedding in $S^3$ determine a canonical diffeomorphism (up to ambient isotopy) $S^3 \setminus \nu L_0^i\xrightarrow{\cong} S^1\times D^2$, since they determine an orientation of the $S^1$ factor and a zero framing. A closed regular neighbourhood $\cl(\nu (L^i\setminus L^i_0))$ of $L^i \setminus L^i_0$ is in the same way canonically diffeomorphic to a disjoint union of solid tori.  Therefore, every link $L^i$ determines an embedding $\hat L^i:\bigsqcup_{k=1}^{n_i}\, S^1\times D^2\hookrightarrow S^1\times D^2$.  These embeddings determine a sequence $T_0 \supset T_1 \supset T_2 \supset \ldots$ where $T_0\subset S^3$ is a single unknotted solid torus and each subsequent term $T_s = \bigcup_{I_s} S^1\times D^2$ is a disjoint union of solid tori, where $I_s := \prod_{i=1}^s n_i$ and $I_0 := 1$. For $s\in\bbN$, the subset $T_s \subset T_{s-1}$ is obtained as
\[T_s=\bigcup_{I_{s-1}}\bigsqcup_{k=1}^{n_s}S^1\times D^2\stackrel{\bigcup\hat L^s}{\longhookrightarrow}\bigcup_{I_{s-1}} S^1\times D^2=T_{s-1}.\]
We define the decomposition $\mcD$ as the connected components of their intersection $\bigcap_{s\in\bbN}T_s$. Such decomposition spaces are called \emph{toroidal}. As above the notions of a link in $S^3$ with a distinguished unknotted component $L_0$ and a link in a solid torus are considered as interchangeable in this paper via $S^3\setminus \nu L_0 \cong S^1\times D^2$.

We note that our results also apply if $S^3$ is replaced by a submanifold of $S^3$ which contains $T_0$.  Isotopies of the defining links can change the actual decomposition, although the homeomorphism type of the quotient $S^3/\mathcal{D}$ remains unaltered.

A good example to keep in mind, and indeed one of the motivating examples which this work aims to generalise, is when $L^i$ is either the Whitehead link or the Borromean rings.  Then in a solid torus we have, respectively, either a Whitehead double of the unknot $S^1 \times \{0\} \subset S^1 \times D^2$ or a Bing double of this unknot.  We say that a decomposition is \emph{pure} if the same link is used at every step of its defining sequence i.e.\ $L^i$ is a fixed link $L$ for all~$i$.  The pure decompositions arising from the Borromean rings and the Whitehead link are referred to as the Bing and Whitehead continua respectively.

The Bing continua are shrinkable~\cite{Bing-shrinking-bing-doubles}; it turns out that this implies that the double of the complement in $S^3$ of the Alexander horned ball (this complement has also been called the Alexander gored ball) is homeomorphic to $S^3$.

The Whitehead continuum $\Wh$ is not shrinkable, although interestingly the decomposition \mbox{$\{\Wh \times \{r\}\,|\, r \in \R\}$} is shrinkable in $S^3 \times \R$~\cite{Andrews-Rubin-65}.

A decomposition which is defined using a combination of these two links is known as a mixed Bing-Whitehead decomposition. These decomposition spaces describe the frontiers of Freedman-Quinn handles, which are constructed in the proof of the 4-dimensional disc embedding theorem in~\cite{FQ}.  For a concise discussion of the relationship between mixed Bing-Whitehead decomposition spaces and convergent grope-disc towers see the introduction of~\cite{Ancel-Starbird-1989}, where F.~Ancel and M.~Starbird give a precise answer to the shrinking question for these spaces.

Now we turn to describing our results, which generalise the results of Ancel and Starbird to decompositions defined using arbitrary links.
To every link $L$ in a solid torus we will associate, in \cref{defn:disc-rep-fns}, a function $D_L\colon\bbN_0\to\bbN_0$ called the \emph{disc replicating function} of $L$.  These functions provide a way to decide whether a decomposition obtained from a sequence of such links is shrinkable.  Now we state our main theorem.

\begin{thmn}[\ref{thmA}]
A decomposition $\mcD$ of $S^3$ obtained from a sequence of links $\{L^i\}_{i \in \bbN}$ is shrinkable if and only if
\[\lim_{p\to\infty} (D_{L^{m+p}}\circ\ldots\circ D_{L^m})(k)=0\]
for all $k,m\in\bbN$.
\end{thmn}

We note that in fact due to~\cite[Theorems~3~and~9]{Armentrout-66}, for the decompositions we consider
the quotient $S^3/\mathcal{D}$ is homeomorphic to $S^3$ if and only if $\mcD$ is shrinkable.  For details see \cref{Remark:not-homeo-to-S3}.

The proof of \cref{thmA} is given in \cref{S:proof}.  As mentioned above, the main idea is to associate, to each link $L$ in a solid torus of the defining sequence, a function $D_L \colon \N_0 \to \N_0$ which we call the disc replicating function of $L$.  To define $D_L$ we need the notion of a $k$-interlacing of a solid torus (Definition~\ref{defn:interlacing-discs}); roughly speaking this is a collection of $2k$ meridional discs labelled alternately $A$ and $B$.  Given a $k$-interlacing the defining property of the function $D_L$ is that it gives the maximal integer $D_L(k)$ for which, after any ambient isotopy of $L$ inside the solid torus, there exists a component $L_j$ of $L$ such that $\cl(\nu L_j)$ has at least a $D_L(k)$-interlacing arising from a subset of its intersections with the $A$ and $B$ discs.  Equivalently, these functions can be defined by the property that with input $k$ they give the minimal nonnegative integer such that there exists an ambient isotopy of $L$ so that all components have at most a $D_L(k)$-interlacing.  Intersections with the $A$ and $B$ discs are then used to control the size of decomposition elements, and thus to show that the Bing shrinking criterion (Theorem~\ref{theorem:bing-shrinking-criterion2}) is either satisfied, or not, as appropriate.

The initial reason for thinking about this was to try to understand and provide some context for the arguments of Bing~\cite{Bing-shrinking-bing-doubles, Bing-62-non-shrinking} and Ancel and Starbird~\cite{Ancel-Starbird-1989} that describe when certain decompositions do and do not shrink; see also R.~Daverman~\cite[Section~9]{Daverman-86}. These arguments were first introduced to the authors in lectures of M.~Freedman on the 4-dimensional disc embedding theorem for the semester on 4-manifolds and their combinatorial invariants hosted by the Max Planck Institute for Mathematics in Bonn.

We reformulate and generalise results of Sher \cite[Theorem~4]{Sher-67} and S.~Armentrout \cite[Theorem~1]{Armentrout-70} on this topic, and as mentioned above we generalise the formula of Ancel and Starbird~\cite{Ancel-Starbird-1989} (also proved later by D.~Wright~\cite{Wright-Bing-Whitehead-1989}) that describes precisely which mixed Bing-Whitehead decompositions shrink.  We give new examples of decompositions for which we can determine whether they shrink.  As far as the authors are aware, our conditions supersede all previously known results on the shrinking or nonshrinking of these toroidal decompositions.

While upper bounds on $D_L$ can be found easily by repositioning the link, in \cref{S:milnor} we will show how to compute lower bounds for $D_L$ using Milnor invariants.  With these methods we can determine the function $D_L$ for every $(n,m)$-link $L$.  An $(n,m)$-link is formed from a meridian of a solid torus and a chain of $n$ unknots inside this solid torus, each of which links the previous and the next with linking number $1$, with the last also linking the first, such that the whole chain has winding number $m$ around the solid torus, and there is no additional entangling of the links in the chain.  Figure~\ref{figure:nmlink} shows a $(4,3)$-link.  The $(n,m)$-links give a nice and large class of examples which we will investigate thoroughly.  Note that the Bing link is a $(2,1)$-link and the Whitehead link is a $(1,1)$-link.

\begin{figure}[h]
\includegraphics[width=6cm]{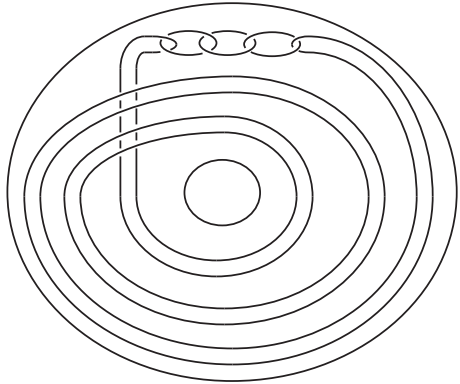}
\caption{A $(4,3)$-link with the meridian of the solid torus omitted.}
\label{figure:nmlink}
\end{figure}

\begin{propn}[\ref{P:nmlink}]
Let $L$ be an $(n,m)$-link.  Then the disc replicating function $D_L$ is given by
\[D_L(k)=\max\{\left\lceil\tfrac{2mk}{n}\right\rceil-1,0\}.\]
\end{propn}

In particular the disc replicating function for the Bing link is $D_L(k) = k-1$, while for the Whitehead link it is $D_L(k)=2k-1$; these functions appear in~\cite{Ancel-Starbird-1989, Wright-Bing-Whitehead-1989}.

From this and \cref{thmA} we deduce the following.  Let $\mathcal{D}$ be a decomposition of $S^3$ arising from a sequence of links $L^1,L^2,L^3,\dots$ where $L^i$ is an $(n_i,m_i)$-link.

\begin{corn}[\ref{C:nmlinks}]
Define~$\tau_i:=n_{i}/2m_{i}$.
 \begin{enumerate}
\item \label{item:cor-5-point-2-item-1-intro} If $\sum_{j=1}^\infty\prod_{i=1}^j\tau_i$ converges, then the decomposition $\mcD$ does not shrink.
\item \label{item:cor-5-point-2-item-2-intro} If $\sum_{j=1}^\infty\tfrac{1}{n_j}\prod_{i=1}^j\tau_i$ diverges, then $\mcD$ does shrink.
\end{enumerate}
In particular we have:
\begin{enumerate}
\setcounter{enumi}{2}
\item \label{item:as-intro} If $\sup_{i\in\bbN} n_{i}<\infty$, then $\mcD$ shrinks if and only if $\sum_{j=1}^\infty\prod_{i=1}^j\tau_i$ diverges.
\item \label{item:cor-5-point-2-item-4-intro} If the sequence of links is periodic; that is if there exists $p\in\bbN$ with $L^i=L^{i+p}$ for all $i\in\bbN$, then $\mcD$ shrinks if and only if $\prod_{i=1}^p\tau_i\geq 1$.
\end{enumerate}
\end{corn}

More background on $(n,m)$-links and the proof of \cref{P:nmlink} and \cref{C:nmlinks} is given in \cref{S:examples}. Furthermore we will use these criteria to show how the afore-mentioned results of Ancel-Starbird, Sher and Armentrout follow as corollaries of \cref{thmA} and \cref{P:nmlink}.

We finish the introduction by noting a possible extension of our work.

\begin{challenge}
  Extend our techniques to deal with decompositions which have defining sequences given by nested handlebodies. See~\cite{Bing-57-handlebody-decomps} for an example.
\end{challenge}

\textbf{Acknowledgements:}
This paper was written while the authors were visitors at the Max Planck Institute for Mathematics in Bonn during the $4$-manifolds semester.  We are grateful to the institute for its support and hospitality, and especially to the organisers of the semester Michael Freedman, Matthias Kreck and Peter Teichner.  We would also like to thank our fellow participants of the $4$-manifolds semester, and especially Stefan Behrens, Robert Edwards, Stefan Friedl, Fabian Hebestreit, Frank Quinn and Peter Teichner for helpful conversations and suggestions.  We also thank the referee for several helpful suggestions.

Finally, the first author is supported by SFB 878 Groups, Geometry and Actions, while the second author gratefully acknowledges an AMS-Simons travel grant which helped with his travel to Bonn.

\tableofcontents

\section{Background and definitions}
\label{S:notations}

\subsection{Shrinking a decomposition}
A \emph{decomposition} of a compact metric space $X$ is a collection $\mathcal{D} = \{\Delta_i\}$ of pairwise disjoint closed subsets of $X$ with $\bigcup\,\mathcal{D} = X$.  Following custom we abuse notation and refer to the decomposition as the elements of $\mathcal{D}$ which are not singletons, with the understanding that once the nontrivial decomposition elements have been specified the rest of the space is decomposed into singleton sets.  We are interested in the topology of the space $X/\mathcal{D}$ obtained by collapsing each $\Delta_i$ to a point.
We say that a decomposition $\mathcal{D}$ of a compact manifold $X$ is \emph{shrinkable} if the quotient map $q \colon X \to X/\mathcal{D}$ can be approximated by homeomorphisms.  That is, there exists a sequence of homeomorphisms which converges to $q$ in the supremum norm.  In particular this implies that $X$ is homeomorphic to $X/\mathcal{D}$.

Note that we may have to appeal to the Urysohn metrisation theorem in order to endow the quotient space with a metric.  We assume that decompositions are such that the quotient is Hausdorff, in order to apply the Urysohn metrisation theorem.  This follows if the decomposition is \emph{upper semi-continuous}~\cite[Pages~8-15]{Daverman-86}, which will always hold for the toroidal decompositions studied in this paper.

Approximating a map by  homeomorphisms is possible if and only if the \emph{Bing Shrinking Criterion} holds, which in R. Edwards' formulation \cite[Section~9]{Edwards-ICM-80} is as follows.

\begin{thm}[Bing Shrinking Criterion]\label{theorem:bing-shrinking-criterion2}
  A surjective map $f \colon X \to Y$ of compact metric spaces can be approximated by homeomorphisms if and only if for any $\eps > 0$, there exists a homeomorphism $h_{\eps} \colon X \to X$ such that the following two conditions are satisfied.
  \begin{enumerate}
    \item \label{item:BSC-1} The homeomorphism $h_{\eps}$ does not move points very far in the metric of $Y$:  $$d_Y(f(x),f \circ h_{\eps}(x)) < \eps$$ for all $x \in X$; and
    \item \label{item:BSC-2} The inverse image sets become sufficiently small under $h_{\eps}$: $$\diam_X (h_{\eps}(f^{-1}(y))) < \eps$$ for all $y \in Y$.
  \end{enumerate}
  \qed
\end{thm}
A detailed proof is given in~\cite[Section~5]{Daverman-86}.  We provide a brief heuristic.
If we can find a sequence~$h_{1/n}$ which converges in the supremum norm (see~\cite[Pages~5~-~7]{ferry-gt-notes} for how to construct such a sequence) then we can see that $f$ can be approximated by homeomorphisms as follows. Let $h_\infty$ be the limit of the sequence of functions $h_{1/n}$ in the supremum norm.  Then $f$ factors through $h_\infty$, as in the diagram below.
\[\xymatrix{X \ar[rr]^{f} \ar[dr]_{h_{\infty}} & & Y \\ & X \ar @{-->}[ur]_{f'} &}\]
Here~$f'$ is defined by $f'(x):=f(h_\infty^{-1}(x))$.  This makes sense because $h_\infty$ and $f$ have the same point inverses by property (\ref{item:BSC-2}), i.e.\ $h_{\infty}(z) = h_{\infty}(z')$ if and only if $f(z) = f(z')$.  The map~$f'$ is a bijection.  It is continuous, by the following argument.  For a closed subset $C\subseteq Y$ we have $(f')^{-1}(C)=h_\infty(f^{-1}(C))$ and we claim this is closed:~$f$ is continuous, so $f^{-1}(C)$ is closed and therefore compact since~$X$ is compact.  Then~$h_\infty$ is continuous, so $h_\infty(f^{-1}(C))$ is a compact subset of a metric space and therefore closed as claimed.  Thus~$f'$ is a continuous bijection between compact Hausdorff spaces and is therefore a homeomorphism.  By construction~$f'\circ h_{1/n}$ approximates~$f$.

The beauty of the Bing Shrinking Criterion is that in order to see that a map can be approximated by homeomorphisms, we do not need to see the existence of homeomorphisms~$h_{\eps}$ which converge; rather for different~$\eps$ they can be constructed independently.

\subsection{Interlacing discs in a solid torus}

The following two definitions appear in Ancel-Starbird~\cite{Ancel-Starbird-1989} and in Wright~\cite[Appendix~A]{Wright-Bing-Whitehead-1989}.

\begin{defi}[Meridional discs]\label{defn:meridional-discs}
A \emph{meridian} of a solid torus $T$ is a simple closed curve in $\partial T$ which bounds a disc in $T$ but not in $\partial T$.  A \emph{meridional disc} of $T$ is a locally flat disc $\Delta \subset T$ such that $\partial T \cap \Delta = \partial \Delta$ is a meridian of $T$.
\end{defi}

\begin{defi}[Interlacing discs]\label{defn:interlacing-discs}
Let $T$ be a solid torus.  Two disjoint collections of pairwise disjoint meridional discs $A = \bigcup_{i=1}^k \,A_i$ and $B = \bigcup_{i=1}^k\, B_j$ for $T$ are called a \emph{$k$-interlacing collection of meridional discs}, if each component of $T \setminus (A \cup B)$ has precisely one $A_i$ and one $B_j$ in its closure.  We make the convention that a $0$-interlacing of meridional discs is the empty set.

We say that two disjoint subsets $A,B \subset T$ form a \emph{$k$-interlacing} for $T$, for $k \geq 1$, if there are subsets $A' \subseteq A$ and $B'\subseteq B$ which form a $k$-interlacing collection of meridional discs for $T$, as in the previous paragraph, such that it is impossible to find such subsets which form a $(k+1)$-interlacing collection of meridional discs for $T$.

\end{defi}

\begin{defi}[Meridional $k$-interlacing]
We call a $k$-interlacing a \emph{meridional $k$-interlacing} if all components of $A$ and $B$ are meridional discs of $T$.
\end{defi}

For a decomposition inside a torus $T$ defined as the intersection of nested tori as above we will use $k$-interlacings for $T$ to measure the size of the nested tori. We will show that the decomposition is shrinkable if and only if there is an ambient isotopy of the nested tori such that for every interlacing there exists a stage such that each torus of this stage meets at most one of the discs of the interlacing. This motivates the next definition.

We will define a function $D_L\colon \bbN_0\to \bbN_0$ called the \emph{disc replicating function of a link $L$}. First we will define functions $U_L,D_L$ which assign a number to each $k$-interlacing.  Then we will show that these functions only depend on $k$ and not on the specific interlacing.  Furthermore we will show that $U_L=D_L$ for all interlacings.


\begin{defi}\label{defn:upper-interlacing-no}
Let $A,B$ be a meridional $k$-interlacing of a torus $T$. For a link $L$ in $T$ we define $U_L(A,B)$ to be the maximal integer $k$ such that for any link $L'$ which is ambient isotopic to $L$ and any closed regular neighbourhood $\cl(\nu L')$ which \emph{intersects $A,B$ only in meridional discs} there is at least one connected component $\cl(\nu L'_j)$ of $\cl(\nu L')$ such that the intersection with $A,B$ gives rise to at least a $U_L(A,B)$-interlacing for the solid torus $\cl(\nu L'_j)$. Since $\cl(\nu L')$ intersects $A,B$ only in meridional discs this will always be a meridional $U_L(A,B)$-interlacing. In particular, $U_L(\emptyset)=0$.
\end{defi}

\begin{defi}\label{defn:lower-interlacing-number}
Let $A,B$ be a meridional $k$-interlacing of a torus $T$. For a link $L$ in $T$ we define $D_L(A,B)$ to be the maximal integer $k$ such that for any link $L'$ which is ambient isotopic to $L$ and any closed regular neighbourhood $\cl(\nu L')$ \emph{whose boundary intersects $A,B$ transversely and only in meridians}, there is at least one connected component $\cl(\nu L'_j)$ of $\cl(\nu L')$ such that the intersection with $A,B$ gives rise to at least a $D_L(A,B)$-interlacing for the solid torus $\cl(\nu L'_j)$. In particular, $D_L(\emptyset)=0$.
\end{defi}

\begin{rem}\label{remark:explaining-difference-between-functions}
  The difference between Definitions~\ref{defn:upper-interlacing-no} and~\ref{defn:lower-interlacing-number} is that for $U_L$, only intersections in meridional discs of $\cl(\nu L')$ are allowed, while for $D_L$ intersections where the boundary $\partial(\cl(\nu L'))$ intersects $A,B$ in meridians are permitted.  In the latter case, for example, there might be an annulus in $(A \cup B) \cap \cl(\nu L')$ such that both boundary curves are meridians of $\cl(\nu L')$.
\end{rem}

Now we show that the interlacing numbers defined above do not depend on the interlacing $A,B$.

\begin{lemma}
\label{lemma:independence}
For any two \emph{meridional} $k$-interlacings $(A,B)$ and $(A',B')$, and any link $L$, the numbers $D_L(A,B)$ and $D_L(A',B')$ agree; moreover the numbers $U_L(A,B)$ and $U_L(A',B')$ also agree.
\end{lemma}

\begin{proof}
We first note that for two collections of $k$ meridional discs $A$ and $A'$ in $T$, by the Sch\"{o}nflies theorem, there exist an orientation preserving homeomorphism $h:T\to T$ such that $h(A)=A'$.

Now let two meridional $k$-interlacings $(A,B)$ and $(A',B')$ and a link $L$ be given. Let $L'$ be a link ambient isotopic to $L$ and let $\nu L'$ be a regular neighbourhood of $L'$ such that the intersection of each component of $\cl(\nu L')$ with $A,B$ gives rise to a (meridional) $m$-interlacing with $m\leq D_L(A,B)$ (respectively $m\leq U_L(A,B)$).  Such an ambient isotopy can always be found by definition of $D_L(A,B)$ (respectively $U_L(A,B)$).  We can push off enough copies of parallel discs of the discs in $A$ and $B$ such that there exists an orientation preserving homeomorphism $h:T\to T$ with $h(A')\subseteq A$ and $h(B')\subseteq B$.  (There might be extra consecutive $A$ or $B$ discs in a given $k$-interlacing, which get deleted in order to form a $k$-interlacing collection of meridional discs.)

Let $E$ be the union of $A$ together with the discs pushed off $A$, and let $F$ be the union of $B$ together with the discs pushed off $B$.  Since we can push off copies in a small neighbourhood, we can achieve that $E,F$ is again a meridional $k$-interlacing and that the intersection of each component of $\cl(\nu L')$ with $E,F$ gives rise to a (meridional) $m$-interlacing with $m\leq D_L(A,B)$ (respectively $m\leq U_L(A,B)$).  By \cite[Proposition 1.10]{Burde-Zieschang:1985-1}, which says that a homeomorphism of the ambient space carrying one link to another is the same as an ambient isotopy between the links, the link $h^{-1}(L')$ is ambient isotopic to $L$ and the intersection of each component of $h^{-1}(\cl(\nu (L')))$ with $A',B'$ gives rise to a (meridional) $m$-interlacing with $m\leq D_L(A,B)$ (respectively $m\leq U_L(A,B)$), since $h(A')\subseteq E$ and $h(B')\subseteq F$. Therefore, $D_L(A',B')\leq D_L(A,B)$ (and $U_L(A',B')\leq U_L(A,B)$). Since the situation is symmetric in $A,B$ and $A',B'$ this proves the lemma.
\end{proof}

\begin{defi}[Disc replicating function]\label{defn:disc-rep-fns}
For a link $L$ in a torus $T$ the functions $U_L\colon\bbN_0\to\bbN_0$ and $D_L\colon\bbN_0\to\bbN_0$ are defined by $U_L(k):=U_L(A,B)$ and $D_L(k):=D_L(A,B)$ where $A,B$ is any meridional $k$-interlacing of $T$, $U_L(A,B)$ is from \cref{defn:upper-interlacing-no} and $D_L(A,B)$ is from \cref{defn:lower-interlacing-number}. These functions are well defined by \cref{lemma:independence} and have the property $D_L(0)=U_L(0)=0$.

By \cref{lemma:interlacing-functions-coincide} below, these two functions coincide.   Thus we call $D_L$ the \emph{disc replicating function for the link $L$}.
\end{defi}

In \cref{lemma:interlacing-functions-coincide} we will make use of the following Technical Lemma of Ancel and Starbird~\cite[Page~301]{Ancel-Starbird-1989}, which we state here for the convenience of the reader.

\begin{lemma}[Ancel-Starbird Technical Lemma] \label{lemma:ancel-starbird-technical-lemma}
  Suppose $P_1,P_2,\dots,P_m$ is a sequence of parallel planes in $\R^3$ such that if $1 \leq i < j <k \leq m$, then $P_j$ separates $P_i$ and $P_k$.  Set $P = P_1 \, \cup \, P_2 \, \cup \, \cdots \, \cup \, P_m$.  Suppose $T$ is a solid torus in $\R^3$ such that $\partial T$ is transverse to $P$, each component of $\partial T \cap P$ is a meridian of $T$, and $T \cap P_i \neq \emptyset$ for $1 \leq i \leq m$.  Then there is a sequence $A_1,A_2,\dots,A_{2m}$ of pairwise disjoint meridional discs of $T$ in cyclic order on $T$ such that $A_i \cup A_{2m+1-i} \subset P_i$ for $1 \leq i \leq m$.
\end{lemma}

In general, even though intersections of the planes $P_i$ with $\partial T$ are always meridians, the intersections with $T$ may be discs, annuli, or discs with holes.  Ancel and Starbird define a notion of the \emph{height} of a component of $P_i \cap \partial T$, which is the maximal integer $h$ such that there is a subset of $P_i \cap \partial T$ which comprises $h$ concentric circles, whose outside circle is the given component of $P_i \cap \partial T$.  If the height is one for all components of $P \cap T$, then all intersections are meridional discs and the proof is straightforward; it is given in the proof of \cref{lemma:consecutive-discs-ancel-starbird-technical}.  Their proof uses induction to lower the number of components with height greater than one.

\begin{lemma}\label{lemma:interlacing-functions-coincide}
  The functions $U_L,D_L \colon \mathbb{N}_0 \to \mathbb{N}_0$ coincide.
\end{lemma}

\begin{proof}
Since, in the definition of $U_L$, only regular neighbourhoods which intersect the interlacing in meridional discs are allowed, we have the inequality $U_L(k)\geq D_L(k)$ for all $k\in\bbN_0$.  It is a priori possible that ambient isotopies which allow disc-with-holes intersections could reduce the number of the induced interlacing.  The current proof shows that this is not possible.

Our aim is to show that $U_L(k) \leq D_L(k)$ for all $k \in \mathbb{N}_0$.  To achieve this we start with a $k$-interlacing of a solid torus $T$ which intersects a regular neighbourhood of the link components in such a way that the intersections with the boundary $\partial\cl(\nu L_i)$, are in meridians, for all $i$.  We want to alter the interlacing so that intersections with $\cl(\nu L_i)$, for all $i$, are meridional discs, without increasing the interlacing number.  Then since the interlacing number is independent of the interlacing, by \cref{lemma:independence}, we will see that $U_L(k) \leq D_L(k)$ for all $k \in \mathbb{N}_0$.

So let $T$ be a solid torus, let $L_i$ be a component of a link $L \subset T$ and let a meridional $k$-interlacing $A,B$ of $T$ be given.  For the rest of this proof we denote $T_1 := \cl(\nu L_i)$.

\begin{claim}
Either $(A \cup B) \cap T_1 = \emptyset$ or there exists at least one meridional disc $\Delta$ in $(A \cup B) \cap T_1$.
\end{claim}

 To prove the claim suppose that $(A \cup B) \cap T_1 \neq \emptyset$.  Let $G$ be a disc of the interlacing which intersects $T_1$.  Look at an innermost circle in $G$, of the intersections of $\partial T_1$ and $G$.  This either bounds a disc inside $T_1$ or outside $T_1$, since it is innermost.  The circle of intersection is a meridian of $T_1$, and so if it bounds a disc in the complement of $T_1$, then a meridian of the knot $L_i$ would be null homotopic in the complement $S^3 \setminus \nu L_i$ of $L_i$.  Thus the innermost circle bounds a meridional disc in $T_1$, as desired.  This completes the proof of the claim.

Choose one meridional disc $\Delta$ in $(A \cup B) \cap T_1$, and let $C \in \{A_i,B_j\}$ be the disc of the $k$-interlacing which gives rise to it.  Starting at the meridian of $\Delta$ on the boundary of $T_1$ and travelling in the direction of the orientation of $L_i$, let $E$ be the disc of the interlacing from which arises the next intersection of $T_1$ with $A \cup B$; since we only have intersections in meridians on the boundary there is a well defined next intersection.  Also let $F$ be the disc of the interlacing from which arises the next intersection in a meridional disc.

\begin{claim}
Either $E=C$ or $E=F$.
\end{claim}

Suppose that $E$ is neither equal to $C$ nor to $F$.  Then the discs $C$ and $F$ give rise to two parallel planes in $\R^3$, thought of as $\R^2 \times \R$, the universal cover of the interior of $T$). Since $E$ intersects $T_1$ in between $C$ and $F$ it gives rise to a parallel plane in $\R^3$ between the other two.  Now cut off $T_1$ at its meridional discs of intersection with $C$ and $F$ and build a new torus as in Figure~\ref{figure:meridional-intersection}.  In Figure~\ref{figure:meridional-intersection}, the additional cylinder added is labelled $Y.$

\begin{figure}[h]
{\psfrag{C}{$C$}
\psfrag{E}{$E$}
\psfrag{F}{$F$}
\psfrag{T}{$T$}
\psfrag{L}{$T_1$}
\psfrag{G}{$Y$}
\includegraphics[width=11cm]{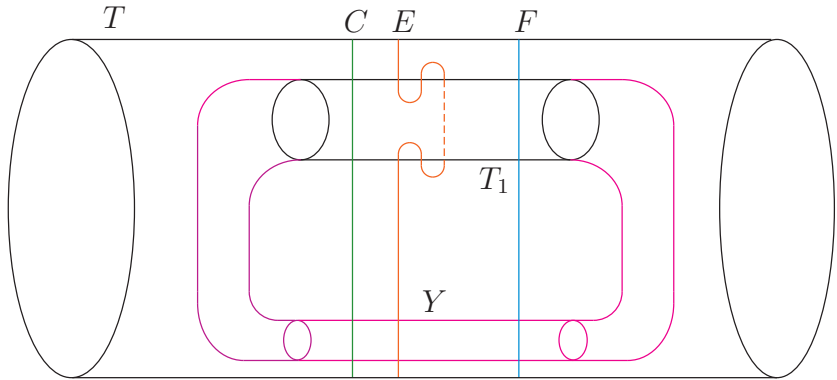}}
\caption{Adding a cylinder $Y$ to build a solid torus intersecting $C$, $E$ and $F$ (at least) twice each.}
\label{figure:meridional-intersection}
\end{figure}

 The additional part of this new torus is embedded in a standard way, and chosen so as to have its intersections with $C,E$ and $F$ to be meridional discs.   This new torus only intersects $E$ in one meridional disc, a contradiction to \cref{lemma:ancel-starbird-technical-lemma}.  Thus we deduce that the claim holds.

Now we change the interlacing of $T$ in the following way.  We cut off $E$ at its intersection with $L_i$ and replace it with a parallel copy of the boundary of $T_1$ and a parallel copy of $C \cap T_1= \Delta$.  See Figure~\ref{figure:replacing-a-meridian} for an indication of how to alter the interlacing.  Here the new interlacing disc is labelled $E'$.

\begin{figure}
{\psfrag{C}{$C$}
\psfrag{E}{$E$}
\psfrag{T}{$T$}
\psfrag{L}{$T_1$}
\psfrag{G}{$E'$}
\includegraphics[width=10cm]{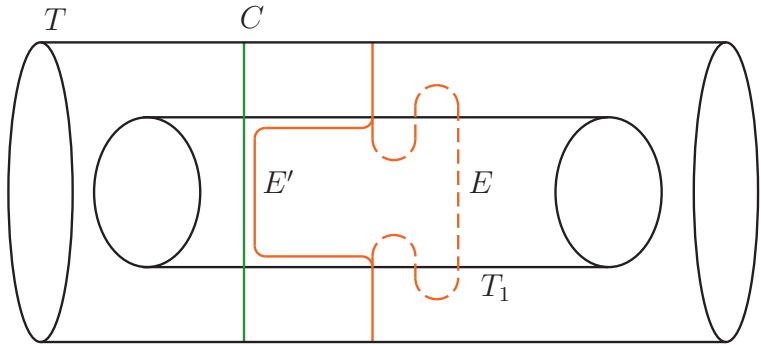}}
\caption{Altering the interlacing disc $E$ to $E'$, so that the intersection $E' \cap T_1$ is a meridional disc.}
\label{figure:replacing-a-meridian}
\end{figure}

The new interlacing of $T$ gives rise to an interlacing of $T_1$ with only one new meridional disc in $T_1$ (and maybe fewer interlacings in $T_1$ or other components of $\nu L$, since $E$ might intersect $\nu L$ elsewhere). But since $T_1$ had a meridional disc coming from $E$ at the same position anyway this new meridional disc only gives rise to an interlacing of $T_1$ of the same number as before.  Also the move has not changed the number of the interlacing of $T$: this is still $k$.  Inductively, by repeating this move as many times as required (which is finitely many times since discs are locally flat and both the discs and tori are compact), we can remove all intersections which are not meridional discs and obtain an induced interlacing of $T_1$ which has at most the same interlacing number as the old induced interlacing of $T_1$. Therefore $U_L(k) \leq D_L(k)$.  Since we already know that $D_L(k) \leq U_L(k)$, we have equality.
\end{proof}

As stated in \cref{defn:disc-rep-fns}, from now on we denote both~$U_L$ and~$D_L$ by~$D_L$.

\section{Proof of the main theorem}

We remind the reader of our main theorem.

\label{S:proof}
\begin{thm}
\label{thmA}
A decomposition $\mcD$ of $S^3$ obtained from a sequence of links $\{L^i\}_{i \in \bbN}$ is shrinkable if and only if
\[\lim_{p\to\infty} (D_{L^{m+p}}\circ\ldots\circ D_{L^m})(k)=0\]
for all $k,m\in\bbN$.
\end{thm}

For the proof we need the following lemma.

\begin{lemma}\label{lemma:intersects-A-and-B-implies-no-BSC}
Suppose that there are nonnegative integers $k,m$ such that for some $r$ and for any $k$-interlacing $A \cup B$ for $(T_{m-1})_r$ we have that for any $s\geq m$ there is a connected component $(T_{s})_{r'}$ of $T_s$ which has nonempty intersection with both of the collections of discs $A$ and $B$. Then $\mathcal{D}$ is not shrinkable.
\end{lemma}

It suffices to have the hypothesis hold for some $r$, but due to the symmetry of the construction of our decompositions, if the hypothesis holds for one $r$ then it holds for all $r$.

\cref{lemma:intersects-A-and-B-implies-no-BSC} follows the strategy employed by Bing~\cite{Bing-62-non-shrinking} and others after him e.g.\ \cite{Sher-67}, \cite{Daverman-86}, \cite{Ancel-Starbird-1989}, \cite{Wright-Bing-Whitehead-1989}.  For the convenience of the reader we provide a proof.

\begin{proof}
We need to show that the Bing shrinking criterion (\cref{theorem:bing-shrinking-criterion2}) does not hold.   Assume, for a contradiction, that it does.  That is, assume the existence of homeomorphisms $h_{\eps}$ with the required properties, for all $\eps$.  Note that $h_{\eps}$ moves points outside of $T_0$ by at most $\delta(\epsilon)$, with $\lim_{\epsilon\to 0}\delta(\epsilon)=0$. Otherwise $h_\epsilon$ would move points too far in the quotient space for the first condition of the Bing shrinking criterion to hold.  Indeed, there is an integer $s$, depending on $\eps$ and tending to infinity as $\eps \to 0$, such that $h_{\eps}$ must be arbitrary close to the identity outside $T_s$ in order to satisfy \cref{theorem:bing-shrinking-criterion2}~(\ref{item:BSC-1}).  We restrict our interest to $\eps$ small enough so that $s \geq m$. In particular, the image $h_\epsilon(T_{m-1})_r$ of $(T_{m-1})_r$ will not shrink as $\epsilon$ goes to zero but will be close to $(T_{m-1})_r$ itself.

We then choose a $k$-interlacing of meridional discs $A \cup B$ in the solid torus $h_\epsilon(T_{m-1})_r$, after the putative homeomorphism $h_{\eps}$ has acted on $S^3$.  The discs $A$ and $B$ must chosen to be sufficiently far apart, so that $d_{S^3}(A,B) > \eps$.  For sufficiently small $\eps$, this is always possible; we only require a contradiction for suitably small values of $\eps$.  We now have two reasons to restrict to small values of $\eps$.  The separation of $A$ and $B$ will imply the existence of a decomposition element which has large diameter, as we now explain.

Note that $h_{\eps}^{-1}(A \cup B)$ is a $k$-interlacing of $(T_{m-1})_r$.  By hypothesis, for every $s\geq m$ there is a component $(T_s)_r$ which intersects both $h_{\eps}^{-1}(A)$ and $h_{\eps}^{-1}(B)$.  Therefore $h_{\eps}((T_s)_r)$ intersects both $A$ and $B$.  By passing to the infinite intersection there must be a decomposition element which intersects both of the subsets $A$ and $B$ in the $k$-interlacing.  As the collections of discs $A$ and $B$ are far apart (their distance apart is bounded below by some $\eps$), that element has diameter at least $\eps$, which contradicts the assumptions on $h_{\eps}$.

We have shown that for sufficiently small $\eps$ there does not exist a homeomorphism $h_{\eps} \colon S^3 \to S^3$ which satisfies the conditions of the Bing shrinking criterion (Theorem~\ref{theorem:bing-shrinking-criterion2}) with respect to $f=q \colon S^3 \to S^3/\mathcal{D}$, and thus that $\mathcal{D}$ is not shrinkable.
\end{proof}

Using this we are now able to prove \cref{thmA}.

\begin{proof}[Proof of \cref{thmA}]

For the only if direction suppose there exist $k,m$ such that $b_s:=(D_{L^s}\circ\ldots\circ D_{L^m})(k)$ is positive for all $s\geq m$. Let $A,B$ be a $k$-interlacing of $(T_{m-1})_r$.

For all $s$, we can perturb $A$ and $B$ keeping them away from any component $(T_s)_r$ they did not intersect before and such that they intersect the boundary of all $T_s$ transversely and only in meridians.  We refer to \cite[Theorem~3]{Bing-62-non-shrinking} and \cite[Proof~of~Lemma~3]{Ancel-Starbird-1989}, where an innermost disc argument is used to discount intersections of $\partial T_s$ with~$A$ and $B$ which are inessential curves in $\partial T_s$ and a small ambient isotopy is used to remove longitudinal intersections.  By~\cite[Theorem~1]{Bing-62-non-shrinking} all intersections are either inessential, a meridian or a longitude.  In general these operations require moving a given $A$ and $B$; first remove inessential and longitudinal intersections of $A$ and $B$ with the tori $T_{m}$, and then proceed inductively.

For the proof of nonshrinking we use the definition of the disc replicating functions $D_{L^i}$ via \cref{defn:lower-interlacing-number}.  By the property of the $D_{L^i}$, at each stage $s\geq m$ of the defining sequence there is always at least one solid torus $(T_s)_{r'}$ for which the intersections with $A$ and $B$ form a $b_s$-interlacing.
In the inductive procedure here, at each stage of the application of the property of disc replicating functions we forget any component of $(A \cup B) \cap (T_s)_{r'}$ which is not a meridional disc.

Since $b_s$ is always positive the assumptions of \cref{lemma:intersects-A-and-B-implies-no-BSC} are satisfied (the assumptions of this lemma are also satisfied for the old interlacing i.e.\ the interlacing before perturbation, since our perturbations did not create any new intersections). Then by \cref{lemma:intersects-A-and-B-implies-no-BSC}, $\mathcal{D}$ does not shrink.\\

Now assume that
\[\lim_{p\to \infty}(D_{L^p}\circ\ldots\circ D_{L^m})(k)=0\]
for all $m, k \in \bbN$. We need to show that the Bing shrinking criterion (\cref{theorem:bing-shrinking-criterion2}) holds.  Let $\eps >0$.  As in Bing's original argument~\cite{Bing-shrinking-bing-doubles}, by going sufficiently deep into the defining sequence, we only need to measure diameter along the $S^1$ direction of the solid torus $T_0$.  Also go sufficiently deep in the defining sequence, to a collection of tori $T_s$, so that as long as we apply a homeomorphism of $S^3$ which is the identity outside of $T_s$, we will always satisfy Theorem~\ref{theorem:bing-shrinking-criterion2}~(\ref{item:BSC-1}).

Look at the collection of tori $T_s\subseteq T_0$. For $k$ large enough we can find meridional $k$-interlacings $A_r,B_r$ for each component $(T_s)_r$ such that each component of $(T_s)_r\setminus (A_r\cup B_r)$ has diameter less than $\eps/2$, measured longitudinally in $T_0$.

For the proof of shrinking we use the definition of the disc replicating functions $D_{L^i}$ via \cref{defn:upper-interlacing-no}.
For a link $L^i$, we may also regard $D_{L^i}(k)$ as giving the minimal integer such that there exists a link $L'$ ambient isotopic to $L^i$ and a regular neighbourhood $\cl(\nu L')$ which intersects a given meridional $k$-interlacing only in meridional discs, and for which all components of $\cl(\nu L')$ have at most a $D_{L^i}(k)$-interlacing arising from their intersections with $A_r,B_r$.  Such an ambient isotopy of $L^i$ determines a homeomorphism of $S^3 \sm \nu L^i_0$ which fixes the boundary, and maps a given regular neighbourhood $\nu L^i \sm L^i_0$ to $\nu L' \sm L'_0$.  Each  connected component of $T_s$ is identified with $S^3 \sm \nu L^s_0$.

Apply, to each connected component of $T_s$, the homeomorphism which the defining property of the disc replicating function $D_{L^s}$ gives to us.  This homeomorphism arranges the components of $T_{s+1}$ so that each of them has a meridional $q$-interlacing for some $q \leq D_{L^{s+1}}(k)$.  Then apply the homeomorphism given to us by the defining property of the disc replicating function $D_{L^{s+2}}$ to the connected components of $T_{s+2}$, and so on.  That is, apply the homeomorphism arising from $D_{L^{s+p}}$ to the connected components of $T_{s+p}$.

Since the sequence $b_p:=(D_{L^{p+s}}\circ\ldots\circ D_{L^{s+1}})(k)$ contains only nonnegative integers, note that converging to $0$ is equivalent to ending with infinitely many zeros.  Thus after finitely many steps we will have a homeomorphism of $S^3$ such that every component of $T_{s'}$, for some $s'$, has a $0$-interlacing from its intersections with $A \cup B$.  Thus each component of $T_{s'}$ intersects at most one of $A$ and $B$, and therefore has diameter less than $\eps$.  Passing to the infinite intersection, the decomposition elements have therefore also been arranged to all have diameter less than $\eps$, so~(\ref{item:BSC-2}) of the Bing shrinking criterion of Theorem~\ref{theorem:bing-shrinking-criterion2} is also satisfied.

\end{proof}

\begin{rem}\label{Remark:not-homeo-to-S3}
If $\mathcal{D}$ shrinks, then $S^3/\mathcal{D}$ is homeomorphic to $S^3$.
As remarked in the introduction, the converse also holds. This was pointed out by Sher~\cite[Preamble~to~Theorem~4]{Sher-67}.  Suppose $\mathcal{D}$ does not shrink.  The decompositions which we consider are \emph{monotone} (that is, the decomposition elements are compact continua), and the image of the nondegenerate elements of~$\mathcal{D}$ under the quotient map $S^3 \to S^3 /\mathcal{D}$ is a compact 0-dimensional set, since it is a subset of some Cantor set.  Therefore by \cite[Theorems~3~and~9]{Armentrout-66} the quotient $S^3/\mathcal{D}$ is not homeomorphic to $S^3$.  We use \cite[Theorem~9]{Armentrout-66} to show that the hypothesis of~\cite[Theorem~3]{Armentrout-66} that the decomposition is \emph{point-like} holds, given that it is monotone, the image of the nondegenerate elements of~$\mathcal{D}$ is a compact 0-dimensional set and the decomposition is definable by 3-cells with handles. A decomposition is \emph{point-like} if the complement of each decomposition element is homeomorphic to $S^3\setminus\{\text{point}\}$.  Armentrout's Theorem~3 says that a point-like decomposition satisfying the assumptions above whose quotient space is homeomorphic to $S^3$ would satisfy the Bing Shrinking Criterion.
\end{rem}

\section{Computable lower bounds via Milnor invariants}
\label{S:milnor}

To show that a decomposition is shrinkable or nonshrinkable, it often suffices to have a sufficiently strong upper or lower bound respectively, for the disc replicating functions $D_L$. For convenience we make the following definition.

\begin{defi}[Upper and lower disc replicating functions]\label{defn:greater-and-lesser-drf}
  We say that a function $f_L \colon \N_0 \to \N_0$ is a \emph{lower disc replicating function} for a link $L$ if $f_L(k) \leq D_L(k)$ for all $k$.
  Similarly we say that a function $g_L \colon \N_0 \to \N_0$ is an \emph{upper disc replicating function} for a link $L$ if $g_L(k) \geq D_L(k)$ for all $k$.
\end{defi}
To construct lower disc replicating functions we will use Milnor invariants.

\subsection{Background on Milnor invariants}\label{subsection:milnor-inv2}

J.~Milnor defined his $\ol{\mu}$-invariants in~\cite{Mil57}.  These (residue classes of) integers $\ol{\mu}_{I}(L)$ are ambient isotopy invariants which are associated to an $n$-component oriented link $L$ and a multi-index~$I$.  For a given $I$, $\ol{\mu}_{I}(L)$ measures the non-triviality of the longitudes of $L$ in a certain lower central series quotient of the link group.  The depth in the lower central series corresponds to the length of $I$.  See e.g.\ \cite{Cochran-derivatives-memoir-90} for a comprehensive study of Milnor invariants.

For the convenience of the reader we now briefly recall the definition of Milnor invariants.  The ensuing exposition follows~\cite[Pages~289--92]{Mil57}. The fundamental group $\pi_1(S^3\setminus\nu L)$ of the link complement is normally generated by choices of meridians $m_1,\dots,m_n$ of the link components. Let $x_1,\dots,x_n$ denote generators of the free group $F$ on~$n$ generators and define $\rho:F\to\pi_1(S^3\setminus \nu L)$ by sending $x_i$ to $m_i$. Let $\lambda_i$ be the zero framed longitude of the component $L_i$ and let $w_i$ be a word in the $x_i$ such that $\rho(w_i) = \lambda_i$.

The beginning of the construction of Milnor invariants is the following theorem.  For a group $G$ we denote its $q$th lower central subgroup by $G_q$; recall that $G_1:=G$ and $G_{q+1}:=[G_q,G]$ for $q \geq 1$.

\begin{thm}[\cite{Mil57} Theorem~4]
The nilpotent quotients of the fundamental group of the exterior of an $n$-component oriented link $L \subset S^3$ are such that:
  \[\pi_1(S^3 \setminus \nu L)/\pi_1(S^3 \setminus \nu L)_q \cong \ll x_1,\dots,x_n\,|\, [x_1,w_1],\dots,[x_n,w_n], F_q\rr.\]
\end{thm}

This means that if the longitudes of the link lie in $F_{q-1}$, i.e.\ $w_i \in F_{q-1}$ for all $i$, then the link group has the same $q$th lower central series quotient as the free group.  Non-vanishing Milnor invariants associated to $I$ of length $q$ measure the failure of the zero framed longitudes to lie in $F_q$.

The \emph{Magnus expansion} of $w_i$ is obtained by substituting $$x_j = 1+\kappa_j \text{ and } x_j^{-1} = \sum_{\ell=0}^{\infty} (-1)^{\ell}\kappa_j^{\ell}.$$
Multiplying out, $w_i$ determines a formal power series in non-commuting variables $\kappa_1,\dots,\kappa_n$.  Let $\mu_{j_1\dots j_s i}(L)$ denote the coefficient of $\kappa_{j_1}\dots \kappa_{j_s}$, so that:
\[w_i = 1+ \sum \mu_{j_1\dots j_s i}(L)\, \kappa_{j_1}\dots \kappa_{j_s}.\]
Equivalently, in terms of the Fox differential calculus: \[\mu_{j_1\dots j_s i}(L) = \phi \left(\frac{\partial^s w_i}{\partial x_{j_1} \dots \partial x_{j_s}} \right),\]
where $\phi \colon \Z[F] \to \Z$ is the augmentation homomorphism.

Let $\Delta_{i_1\dots i_r}(L)$ denote the greatest common divisor of all integers of the form $\mu_{j_1\dots j_p}(L)$, where $2 \leq p < r$, and where $j_1\dots j_p$ ranges over all multi-indices obtained by deleting one or more of the indices from $i_1\dots i_r$ and permuting those which remain cyclically.

Let $\ol{\mu}_{i_1\dots i_r}(L)$ denote the residue class of $\mu_{i_1\dots i_r}(L)$ modulo $\Delta_{i_1\dots i_r}(L)$.

\begin{thm}[\cite{Mil57} Theorem~5]
  For $r \leq q$, the residue classes $\ol{\mu}_{i_1\dots i_r}(L) \in \Z_{\Delta_{i_1\dots i_r}(L)}$ are ambient isotopy invariants of $L$.  
\end{thm}


\subsection{Computable lower bounds}

In this section we define a lower disc replicating function $f_L \colon \N_0 \to \N_0$ (as in Definition~\ref{defn:greater-and-lesser-drf}), associated to an oriented link~$L$ with $L_0$ unknotted, which bounds the link's disc fertility from below.  More precisely, recall that $f_L$ should satisfy the property that if a component of a solid torus $T_s$ has a $k$-interlacing, and the next stage of the defining sequence is determined by $L$, then there is at least one component $L_j$ of $L\setminus L_0$ for which the intersections of the $k$-interlacing discs with $\cl(\nu L_j)$ give rise to an $h$-interlacing for the solid torus $\cl(\nu L_j)$, for some  $h \geq f_L(k)$.

Let $L=L_0 \,\sqcup\, L_1 \,\sqcup\, \dots \sqcup \, L_m$ be an $m$-component oriented link in $S^3$ where~$L_0$ is unknotted.  From this, we produce another link $J$ by some choice of the following sequence of operations.

\begin{enumerate}
  \item \label{item:move1} Take the $d$-fold branched cover of $S^3$ with branching set $L_0$, and let $\wt{J}$ be the pre-image of $L$, where $\wt{J}_0$ is the pre-image of $L_0$.  This is again a link in $S^3$. \\
      \item \label{item:move2} Take a sublink $\widehat{J}$ of $\wt{J}$ which includes $\wt{J}_0$ as $\widehat{J}_0$.\\
     \item \label{item:move3} Blow down (perform $\pm 1$ Dehn surgery) along $\ell$ unknotted components of $\widehat{J} \setminus \widehat{J}_0$, each of which lies in an open $3$-ball in $S^3 \setminus \nu \widehat{J}_0$.  Call the resulting link $J$, with $\widehat{J}_0$ becoming $J_0$.
\end{enumerate}

For $k = 0$ define $f_L^J(k) =0$.  If $J$ is such that no multi-index $I$ exists which contains at least one zero and for which the corresponding Milnor invariant is nonzero, then we define:
\[f_L^J(k) := 0\]
for all $k \in \mathbb{N}$.

Now suppose that a link $J$ can be produced from $L$ with $\ol{\mu}_I(J) \neq 0$ for some multi-index $I$ which contains at least one $0$.  If there is such a multi-index with $|I|=2$, i.e.\ $I=(0j)$ or $I=(j0)$ for some $j$, then we define:
\[f_L^J(k) := \ol{\mu}_I(J)k\]
for all $k \in \N$.  Let $n+1$ be the number of components of $J=J_0 \,\sqcup\, J_1 \,\sqcup \dots \sqcup \, J_n$.  If $|I| > 2$, then we define:
\[f_L^J(k) := \left\lceil\frac{2dk}{n+\ell}\right\rceil - 1\]
for all $k \in \N$.  Finally define the function:
\[f_L(k) := \max\{f_L^J(k)\, | \, J\text{ reached from }L \text{ by operations (\ref{item:move1}), (\ref{item:move2}) and (\ref{item:move3})}\}.\]

\begin{rem}\label{remark:optimal-functions}
  In practice, it is usually not necessary to find the function $f_L$ precisely, only to find a $J$ which gives a sufficiently large lower bound.  Even if we did find $f_L$ precisely, it still may not equal $D_L$.  Nevertheless, this will not overly concern us since if a sequence of integers defined using functions $f_L$ never reaches zero, then neither does a sequence defining using the functions $D_L$.

  Having said that we actually will be able to determine $D_L$ precisely for $(n,m)$-links in \cref{S:examples}.
\end{rem}

Now we show that the functions $f_L$ defined above are indeed lower disc replicating functions.

\begin{thm}[Lower disc replicating functions]\label{theorem:disc-replicating-functions}
  Suppose $T = S^3 \setminus \nu L_0$ has a meridional $k$-interlacing $A,B$.  Then for any link $L'$ which is ambient isotopic to $L$ and such that $\cl(\nu L')$ intersects $A,B$ only in meridional discs, there is a component $L'_j$ of $L'$ such that $\cl(\nu L'_j)$ has an $h$-interlacing arising from its intersections with the $k$-interlacing for $T$, for some $h \geq f_L(k)$.
\end{thm}

\begin{rem}\label{remark:explanation-of-thm-lower-drfs}
  \cref{theorem:disc-replicating-functions} implies that $f_L(k) \leq U_L(A,B) = U_L(k) = D_L(k)$, by \cref{defn:upper-interlacing-no}, \cref{lemma:independence} and \cref{lemma:interlacing-functions-coincide} respectively.
\end{rem}

\begin{proof}[Proof of \cref{theorem:disc-replicating-functions}]
First observe that it suffices to prove the result for a given choice of link $J$ obtained from $L$ by the operations (\ref{item:move1}), (\ref{item:move2}) and (\ref{item:move3}) above.
Since the Milnor invariants which define $f_L$ are ambient isotopy invariants, and since an ambient isotopy of $L$ induces ambient isotopies of $\widehat{J}$ and $J$, the conclusion also holds for any $L'$ ambient isotopic to $L$.

For the case that there is no multi-index $I$ for which $\ol{\mu}_I(J) \neq 0$, so that $f_L^J(k)=0$ for all $k$, the theorem is immediate.

Let $\wt{T}$ be the $d$-fold cover of the solid torus $T$ with covering map $\pi \colon \wt{T} \to T$.  A $k$-interlacing for $T$ lifts to a $kd$-interlacing for $\wt{T}$.  We will refer to the discs of this $kd$-interlacing as $\wt{A}= \wt{A}_1 \, \sqcup\dots\sqcup\, \wt{A}_{kd}$ and $\wt{B}= \wt{B}_1 \,\sqcup\dots\sqcup\, \wt{B}_{kd}$.  In this proof we forget about any extra discs and focus on a subset of the interlacing which is an interlacing collection of meridional discs (recall \cref{defn:interlacing-discs}).

In the case that $|I|=2$, so that we can assume $I=(0j)$ for some $j$, the definition of linking number implies that every meridional disc in the collections $\wt{A}$ and $\wt{B}$ must intersect $J_j$ at least $\ol{\mu}_{0j}(J)$ times all with the same intersection number.  Each intersection gives rise to an intersection with $\cl(\nu J_j)$ which by hypothesis must be a meridional disc.  Ignoring any other intersections, this translates to a $\ol{\mu}_{I}(J)kd$-interlacing, which at least descends to a $\ol{\mu}_{I}(J)k$-interlacing of $\cl(\pi(\nu J_j))$.  The function $f_L^J(k) = \ol{\mu}_I(J)k$, when the Milnor invariant has length $2$, will offer the sharpest possible bound when $d=1$.

We now turn to the case that $|I| >2$.  After performing operations (\ref{item:move1}) and~(\ref{item:move2}) we have a link $\widehat{J}$.  Not including $\widehat{J}_0$, we have $n+\ell$ components.

Define a function $[\,\cdot\,]_{1/2} \colon \R \to \frac{1}{2}\Z$ by $[x]_{1/2} := \frac{1}{2}\lceil 2x\rceil$.  The effect is to round up to the nearest half integer.

Our aim is to show that some component of $\widehat{J}\setminus \widehat{J}_0$ intersects at least $[kd/(n+\ell)]_{1/2}$ pairs of consecutive discs in the $kd$-interlacing of $\wt{T}$.  We have the following lemma, in which perhaps unsurprisingly we interpret half a pair of discs to mean a single disc.  We do not specify whether the extra disc is an $\wt{A}$ or a $\wt{B}$ disc.  We also regard a negative number of pairs of discs as zero discs.

Let $T$ be a solid torus and let $T_{\infty} \approx \R \times D^2$ be the infinite cyclic cover of $T$.  Let $S \subset T$ be an embedded solid torus and let $\wt{S}$ be its pre-image in $T_{\infty}$.

\begin{lemma}\label{lemma:consecutive-discs-ancel-starbird-technical}
  Suppose that $\wt{S}$ is such that $\partial \wt{S}$ has nonempty intersections with each disc in $r \in \frac{1}{2}\Z$ consecutive pairs of meridional discs $P$ for $T_{\infty}$, which are a subset of the pre-images of an $m$-interlacing collection of meridional discs for $T$, for some $m$.  Suppose that each component of $\wt{S} \cap P$ is a meridional disc of $\wt{S}$.  Moreover suppose that the winding number of $S$ in $T$ is zero.  Then the intersections of $S$ with the $m$-interlacing collection of meridional discs for $T$ give rise to an $n$-interlacing of $S$ for some $n \geq 2r-1$.
\end{lemma}

\begin{proof}[Proof of Lemma~\ref{lemma:consecutive-discs-ancel-starbird-technical}]
Since $S$ has winding number zero the lift~$\wt{S}$ is again a solid torus in $T_{\infty}$.  Forgetting the $A$ and $B$ labels we denote the $r$ pairs of meridional discs by $P_1 \cup P_2 \cup \cdots \cup P_{2r}$.  The interior $\Int(T_{\infty})$ is homeomorphic to $\R^3$ in such a way that the discs $P$ are sent to disjoint parallel planes $\{p_i\} \times \R^2$, $p_1 < p_2 < \dots < p_{2r}$ which we also denote by $P_i$.  By hypothesis every intersection of $\wt{S}$ with a plane $P_i$ is a meridional disc of $\wt{S}$ and therefore can contribute to an interlacing of $S$.

The proof is now the same as the (straightforward) height one case of the Technical Lemma of Ancel and Starbird~\cite[Page~301]{Ancel-Starbird-1989}, stated as our \cref{lemma:ancel-starbird-technical-lemma}.  For the convenience of the reader we give it here.  Choose a simple closed curve $\gamma$ starting at a point $q$ on $\partial \wt{S}$ with first coordinate in $\R^3$ less than $p_1$, which intersects each component of $P_i \cap \partial \wt{S}$ transversely in a single point.  Let $u$ be a point on $\gamma$ with first coordinate in $\R^3$ greater than $p_{2r}$.  For $i=1,\dots,2r$, let $C_i$ be the meridional disc in $\wt{S} \cap P_i$ which one meets first while traversing $\gamma$ from $q$ to $u$ and let $C_{4r+1-i}$ be the meridional disc in $\wt{S} \cap P_i$ which one meets first while traversing $\gamma$ from $u$ to $q$.  The meridional discs $C_1,\dots, C_{4r}$ for $\wt{S}$, with their appropriate $A$ and $B$ labels reinstated, give rise to a $(2r-1)$-interlacing of $\wt{S}$ and therefore their images in $T$ give rise to a $(2r-1)$-interlacing for $S$.
\end{proof}

Continuing the proof of Theorem~\ref{theorem:disc-replicating-functions}, we make the following claim.

\begin{claim}
  Some component of $\widehat{J}' := \widehat{J}\setminus \widehat{J}_0$ intersects at least $[kd/(n+\ell)]_{1/2}$ pairs of consecutive discs in the $kd$-interlacing of $\wt{T}$.
\end{claim}

Assuming that the claim holds, the lift of that component to $T_{\infty}$ also intersects at least $[kd/(n+\ell)]_{1/2}$ pairs of discs.  Then Lemma~\ref{lemma:consecutive-discs-ancel-starbird-technical} implies that this component has at least a $2[kd/(n+\ell)]_{1/2}-1 = \lceil 2kd/(n+\ell) \rceil - 1 = f_L^J(k)$-interlacing, which is what we want to show.  Lifting the intersections to $T_{\infty}$ we see that the hypotheses of Lemma~\ref{lemma:consecutive-discs-ancel-starbird-technical} apply. Then observe that intersections of $\cl(\nu \widehat{J}')$ with $\wt{A}$ and $\wt{B}$ descend to similar intersections of $\cl(\pi(\nu \widehat{J}'))$ with $A$ and $B$; once we know that the requisite interlacing arises in $T_{\infty}$ we can deduce its existence in $\wt{T}$ and in $T$, since there is a tower of covering spaces $T_{\infty} \to \wt{T} \to T$.

Thus it remains to prove the claim that some component of $\widehat{J}' = \widehat{J}\setminus \widehat{J}_0$ intersects at least $[kd/(n+\ell)]_{1/2}$ pairs of consecutive discs in the $kd$-interlacing of $\wt{T}$.

Since each blown down component $\widehat{J}'_p$ is contained in a $3$-ball in $\wt{T}$ and is unknotted, it also bounds a disc in that $3$-ball.  This 3-ball can be shrunk so that it misses all discs $C$ of $\wt{A} \cup \wt{B}$ for which $\widehat{J}'_p \cap C = \emptyset$.
Blowing down along the component $\widehat{J}'_p$ can be realised by twisting along a disc whose boundary is $\widehat{J}'_p$.  Thus if a meridional disc of $\wt{T}$ misses $\widehat{J}'$ then there is an ambient isotopy of $J \setminus J_0$ such that after the ambient isotopy we have that $J \setminus J_0$ misses that meridional disc.  However, this is not possible, as the non-vanishing of a Milnor invariant $\ol{\mu}_I(J)$ with at least one zero in the multi-index implies that the longitude $\lambda_0$ of $J_0$ is nontrivial in $\pi_1(S^3 \setminus \nu J)/\pi_1(S^3 \setminus \nu J)_q$ for $q=|I|$.  Therefore each meridional disc of $\wt{A} \cup \wt{B}$ must intersect at least one component of $\widehat{J}'$.

So, we have an $(n+\ell)$-component link $\widehat{J}'$ in $\wt{T}$.  Suppose that each component hits fewer than $[kd/(n+\ell)]_{1/2}$ pairs of meridional discs of $\wt{T}$; that is, at most $[kd/(n+\ell)]_{1/2} - 1/2$.  We will show that this implies that fewer than $kd$ pairs of discs in total can be intersected by the link $\widehat{J}'$.  Then the contrapositive of this implication coupled with our knowledge from the previous paragraph that every disc is intersected, implies the claim.

To see that fewer than $kd$ pairs of discs are intersected, we note that:
\[\left[\frac{kd}{n+\ell}\right]_{1/2} - \frac{1}{2} < \frac{kd}{n+\ell} + \frac{1}{2} - \frac{1}{2} = \frac{kd}{n+\ell},\]
so that
\[(n+\ell)\left(\left[\frac{kd}{n+\ell}\right]_{1/2} - \frac{1}{2}\right)< (n+\ell)\frac{kd}{n+\ell} = kd.\]
The left hand side of the last inequality is the maximum number of pairs of meridional discs in the interlacing $\wt{A} \cup \wt{B}$ which $\widehat{J}'$ can intersect given the assumption that each component intersects fewer than $[kd/(n+\ell)]_{1/2}$ pairs of meridional discs.
This completes the proof of \cref{theorem:disc-replicating-functions}.
\end{proof}

\begin{rem}
  We point out that we could conceivably use some method other than blow downs and Milnor invariants to see that there exists a component of $\widehat{J}$ which intersects each meridional disc.  However this method lends itself nicely to geometric computation, as we will see in the next section. Moreover using Milnor invariants means that Theorem~\ref{thmA} can be applied to vast classes of examples, whereas previous results in the literature focused on certain special links.

  The examples in the next section constitute a large class of links, but we note that the class of links with nonvanishing Milnor invariants is of course much larger, and the lower disc replicating functions defined in this section can be applied to decompositions constructed using these as well.
\end{rem}

\section{Examples: $(n,m)$-links}
\label{S:examples}
 Let $n \ge 1$ and $m \ge 1$.  Suppose we have a chain of $n$ unknots in $T_0$, each of which links the next in the chain with linking number $\pm 1$, such that the last knot of the chain (in the world outside topology it would be called a `link' of the chain) also links the first with linking number $\pm 1$, and such that the whole chain travels around $T_0$ with winding number $m$.  There should be no additional entangling of the links in the chain.  If $n=1$ then the knot clasps itself after winding $m$ times around $T_0$.  Taking the union of the resulting link with a meridian of $T_0$, we obtain an \emph{$(n,m)$-link}.  See Figure~\ref{figure:nmlink} for a picture of a $(4,3)$-link.

Note that the single defining link of the Bing decomposition of~\cite{Bing-shrinking-bing-doubles} is a $(2,1)$-link, while Bing's example of a decomposition which does not shrink~\cite{Bing-62-non-shrinking},~\cite[Chapter~9,~Example~6]{Daverman-86},~\cite[Page~416]{Freedman-JDG-82} has a $(2,2)$-link as its defining link.

Let $L$ be an $(n,m)$-link.  The $m$-fold covering space of $T_0$ contains $m$ copies of a chain of length $n$ with winding number $1$ in $\wt{T}_0$, i.e.\ there is an $(n,1)$-link $\widehat{J}$ as a sublink of $\wt{J}$; see Figure~\ref{figure:nmlink2}.

\begin{figure}[H]
\includegraphics[width=6cm]{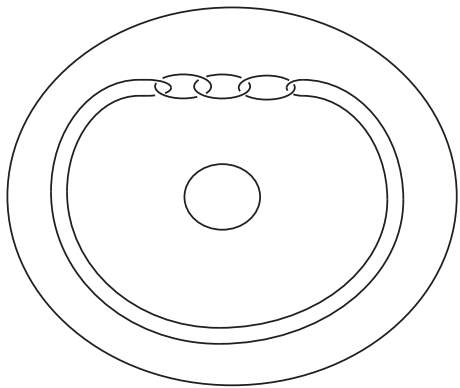}
\caption{A $(4,1)$-link as a sublink of the $3$-fold covering link $\wt{J}$ of the $(4,3)$-link from Figure~\ref{figure:nmlink}.}
\label{figure:nmlink2}
\end{figure}

Assume $n \geq 2$.  Perform $(n-2)$ blow downs on this link, to obtain the $2$-component link of the Bing decomposition inside the solid torus; see Figure~\ref{figure:nmlink3}.  The link $J$ is the 3-component link obtained by including a meridian of the solid torus.  Note that $J$ is the Borromean rings (or the Borromean rings with a clasp changed, depending on the signs of the original linking numbers and of the blow downs) and so has $\ol{\mu}_{012}(J) = \pm 1$.  Even if a clasp is changed so that two of the components have linking number $2$, the length 3 Milnor invariant is still well-defined mod~2.

\begin{figure}[h]
\includegraphics[width=6cm]{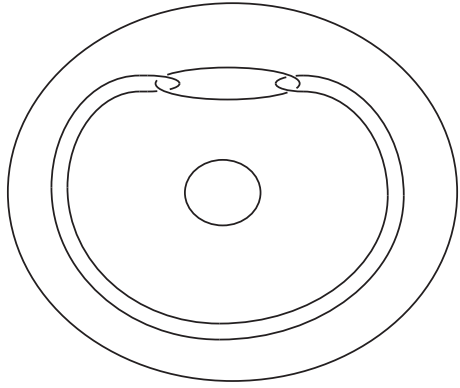}
\caption{The link $J$ obtained by blowing down two of the components of the link of Figure~\ref{figure:nmlink2}.  With the meridian of the solid torus, this is the Borromean rings with a clasp changed.}
\label{figure:nmlink3}
\end{figure}

\begin{prop}
\label{P:nmlink}
For an $(n,m)$-link $L$ the disc replicating functions $D_L$ is given by
$$D_L(k)=\max\{\lceil \tfrac{2mk}{n}\rceil -1,0\}.$$
\end{prop}
\begin{proof}
For $n >1$, to see that $D_L(k)\geq \max\{\lceil \tfrac{2mk}{n}\rceil -1,0\}$ construct $J$ by taking an $m$-fold cover, and as described above blow down $n-2$ components.  This leaves $2$ components inside the solid torus, and the meridian of the solid torus as $J_0$. By \cref{theorem:disc-replicating-functions} we have $f_L^J(k)=\max\{\lceil \tfrac{2mk}{n}\rceil -1,0\}$.  If $n=1$, then a sublink $J$ of the $m$-fold cover is the Whitehead link.  The Whitehead link has $\ol{\mu}_{0011}(J) = \pm 1$, so we obtain the lower disc replicating function $f_L^J(k) = 2mk -1$.  This proves the lower bound.

As shown in an example in Figure~\ref{figure:nmlink4} it is not too hard to isotope the link so that every component inherits at most a $\big(\lceil \tfrac{2mk}{n}\rceil -1\big)$-interlacing.  Since we can arrange that every intersection of the regular neighbourhood of the link with the meridional $k$-interlacing is again a meridional disc, this is an upper bound for $D_L(k)$.  This completes the proof.
\end{proof}

\begin{figure}[h]
{\psfrag{A}{$A$}
\psfrag{B}{$B$}
\includegraphics[width=9cm]{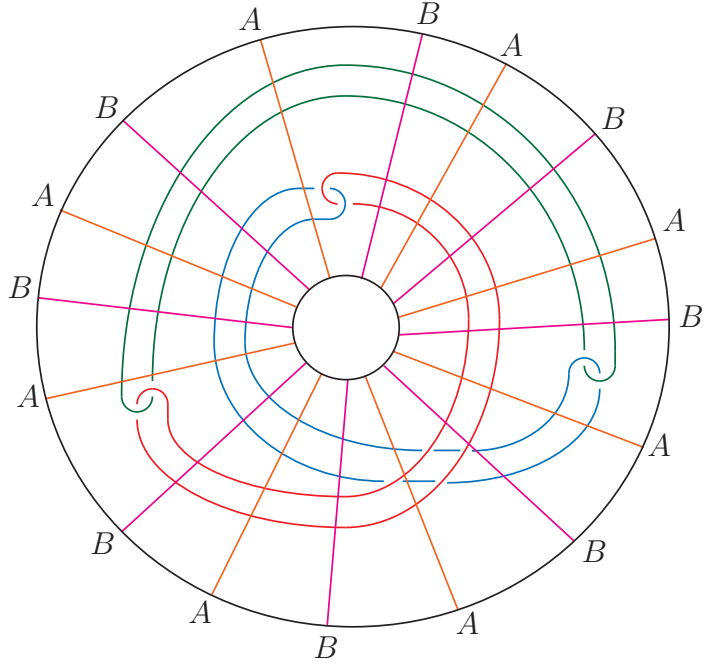}}
\caption{A $(3,2)$-link with an $8$-interlacing.  In this case $D_L(8)=10$.}
\label{figure:nmlink4}
\end{figure}

Combining \cref{thmA} and \cref{P:nmlink} we get the following criteria for shrinking or nonshrinking of a decomposition arising from a sequence of $(n_i,m_i)$-links.  The answer is particulary nice when the sequence of links is periodic, and still quite nice when $\sup_{i\in\bbN} n_{i}<\infty$.

\begin{cor}
\label{C:nmlinks}
Let $L^i$ be an $(n_i,m_i)$-link. Define $\tau_i:=\frac{n_{i}}{2m_{i}}$.
\begin{enumerate}
\item \label{item:cor-5-point-2-item-1} If $\sum_{j=1}^\infty\prod_{i=1}^j\tau_i$ converges, then the decomposition $\mcD$ does not shrink.
\item \label{item:cor-5-point-2-item-2} If $\sum_{j=1}^\infty\tfrac{1}{n_j}\prod_{i=1}^j\tau_i$ diverges, then $\mcD$ does shrink.
\end{enumerate}
In particular we have:
\begin{enumerate}
\setcounter{enumi}{2}
\item \label{item:as} If $\sup_{i\in\bbN} n_{i}<\infty$, then $\mcD$ shrinks if and only if $\sum_{j=1}^\infty\prod_{i=1}^j\tau_i$ diverges.
\item \label{item:cor-5-point-2-item-4} If the sequence of links is periodic; that is if there exists $p\in\bbN$ with $L^i=L^{i+p}$ for all $i\in\bbN$, then $\mcD$ shrinks if and only if $\prod_{i=1}^p\tau_i\geq 1$.
\end{enumerate}
\end{cor}
\begin{proof}~
\begin{enumerate}
\item  Assume that $\sum_{j=1}^\infty\prod_{i=1}^j\tau_i$ converges. Let $g_i\colon \Q \to \Q$ be defined by $g_i(k)=\tau_i^{-1}k-1$ and choose $k_0>\sum_{j=1}^\infty\prod_{i=1}^j\tau_i$.  Such an integer $k_0$ exists since we assume that the right hand side converges.
Then for all $r \ge 1$ we compute that
\begin{eqnarray*}
g_{r}\circ\ldots\circ g_{1}(k_0)&=&\bigg(\prod_{i=1}^r\tau_i^{-1}\bigg)k_0-\sum_{j=1}^r\prod_{i=j+1}^r\tau_i^{-1}\\
&=&\bigg(\prod_{i=1}^r\tau_i^{-1}\bigg)\bigg(k_0-\sum_{j=1}^r\prod_{i=1}^j\tau_i\bigg)\\
&>&0.
\end{eqnarray*}
By \cref{P:nmlink}, for all $k>0$,
and in particular for $k=k_0$, we have that $D_{L^i}(k)=\lceil \tau_i^{-1}k\rceil-1\geq g_i(k)$, so by \cref{thmA} the above inequality therefore implies that $\mcD$ does not shrink.
\item For any fixed $k_0,m\in\bbN$, let $r$ be such that $$\sum_{j=m}^{r+m}\frac{1}{n_{j}}\prod_{i=1}^j\tau_i>\bigg(\prod_{i=1}^{m-1}\tau_i\bigg)k_0.$$
The right hand side is now a fixed integer.  Since we are assuming that the series $\sum_{j=1}^\infty\tfrac{1}{n_j}\prod_{i=1}^j\tau_i$ diverges, so does the sequence $\sum_{j=m}^\infty\tfrac{1}{n_j}\prod_{i=1}^j\tau_i$.  Therefore there exists a partial sum larger than any given integer.

Now define $g_i(k):=\tau_i^{-1}k-\tfrac{1}{n_{i}}$. Then for all $r \ge 0$ we compute that:
\begin{eqnarray*}
g_{r+m}\circ\ldots\circ g_{m}(k_0)&=&\bigg(\prod_{i=m}^{r+m}\tau_i^{-1}\bigg)k_0-\sum_{j=m}^{r+m}\bigg(\tfrac{1}{n_{j}}\prod_{i=j+1}^{r+m}\tau_i^{-1}\bigg)\\
&=&\bigg(\prod_{i=1}^{r+m}\tau_i^{-1}\bigg)\bigg(\bigg(\prod_{i=1}^{m-1}\tau_i^{-1}\bigg)k_0-\sum_{j=m}^{r+m}\bigg(\tfrac{1}{n_{j}}\prod_{i=1}^j\tau_i\bigg)\bigg)\\
&<&0.
\end{eqnarray*}
By \cref{P:nmlink}, for all $k>0$ we have $D_{L^i}(k)=\lceil \tau_i^{-1}k\rceil-1\leq \max\{g_i(k),0\}$, so by \cref{thmA} the above inequality therefore implies that $\mcD$ does shrink.
\item Let $B:=\sup_{j\in\bbN}n_{j}$, then $$\sum_{j=1}^\infty\tfrac{1}{n_{j}}\prod_{i=1}^j\tau_i > \frac{1}{B}\sum_{j=1}^\infty\prod_{i=1}^j\tau_i$$
and therefore $\sum_{j=1}^\infty\tfrac{1}{n_{i}}\prod_{i=1}^j\tau_i$ diverges if $\sum_{j=1}^\infty\prod_{i=1}^j\tau_i$ diverges, so $\mathcal{D}$ shrinks by (\ref{item:cor-5-point-2-item-2}).  The only if direction is immediate.
\item In this case $\sup_{i\in\bbN}n_{i}<\infty$ and so (\ref{item:as}) applies.  Since $\tau_i$ is also periodic with period $p$, we have that:
$$\sum_{j=1}^\infty\prod_{i=1}^j\tau_i=\bigg(\sum_{j=1}^p\prod_{i=1}^j\tau_i\bigg)\sum_{j=0}^\infty\bigg(\prod_{i=1}^{p}\tau_i\bigg)^j$$
and $\sum_{j=0}^\infty\left(\prod_{i=1}^{p}\tau_i\right)^j$ diverges if and only if $\prod_{i=1}^{p}\tau_i\geq 1$.
\end{enumerate}
\end{proof}

\subsection{The results of Ancel and Starbird}
Let $w_0<w_1<w_2<\ldots$ be a sequence of positive integers.
Let $L^{w_i}$ be the Whitehead link and let $L^j$ be the Borromean rings for $j\notin\{w_i\mid i\in\bbN\}$. Then the associated links in the solid torus are the Bing and Whitehead doubles of the unknot respectively. Let $\mcD$ be the associated decomposition. By \cref{P:nmlink} a disc replicating function for the Borromean rings is $D_1(k) = k-1$ and a disc replicating function for the Whitehead link is $D_2(k) = 2k-1$. These agree with the formulae in Lemma~4 and Lemma~5 of \cite{Ancel-Starbird-1989}.

Let $c_1:=w_1-1$ and $c_i:=w_i-w_{i-1}-1$ for $i>1$. These denote the number of Bing/Borromean links in between successive Whitehead links.

The following theorem was proved by F.~Ancel and M.~Starbird~\cite{Ancel-Starbird-1989}, and also later by D.~Wright~\cite{Wright-Bing-Whitehead-1989}.

\begin{thm}[Ancel, Starbird]\label{theorem:ancel-starbird}
The decomposition $\mcD$ shrinks if and only if   $\sum_{i=1}^{\infty} \tfrac{c_i}{2^i}$ diverges.
\end{thm}
The proof follows easily from \cref{C:nmlinks}.
\begin{proof}
Let $\tau_{w_i}=\tfrac{1}{2}$ and let $\tau_\ell=1$ for $\ell\notin\{w_i\mid i\in\bbN\}$, as is consistent with the definition of the $\tau_j$ in \cref{C:nmlinks}. By \cref{C:nmlinks}~(\ref{item:as}) the decomposition $\mcD$ shrinks if and only if $\sum_{j=1}^\infty\prod_{i=1}^j\tau_i= 2 \sum_{j=1}^\infty\tfrac{c_j}{2^j}$ diverges.
\end{proof}

\subsection{The results of Sher and Armentrout}
Sher's theorem~\cite[Theorem~4]{Sher-67} was generalised by Armentrout~\cite[Theorem~1]{Armentrout-70} as follows.  Let $L^{i}$ be an $(n_i,m_i)$-link for $i \in \N$. As always let $\mathcal{D}$ denote the associated decomposition space of $S^3$.

\begin{thm}[Armentrout, Sher]\label{theorem:armentrout-sher}
Suppose that $n_i < 2m_i$ for all $i$.  Then $\mathcal{D}$ does not shrink.
\end{thm}

\begin{proof}
By \cref{P:nmlink} the disc replicating function $D_{L^i}$ of $L^i$ is given by
\[D_{L^i}(k)=\left\lceil\tfrac{2m_i k}{n_i}\right\rceil-1\]
Since $n_i < 2m_i$ it follows that $\lceil\tfrac{2m_i k}{n_i}\rceil > k$ and therefore $D_{L^i}(k)\geq k$. By \cref{thmA} this implies that $\mcD$ does not shrink.
\end{proof}

\subsection{More examples of $(n,m)$-links}
In this subsection we give two examples of mixed decomposition of $(n,m)$-links for which neither criterion of \cref{C:nmlinks} applies. The first does not shrink, the second does. This shows that \cref{C:nmlinks}~(\ref{item:cor-5-point-2-item-1}) and (\ref{item:cor-5-point-2-item-2}) are not sharp.  In these examples both $n_i$ and $m_i$ tend to infinity as $i$ tends to infinity.

\begin{example}
For our first example, let $L^i$ be a $(2i,i+1)$-link. Then $\tau_i=\tfrac{i}{i+1}$ and we have $$\sum_{j=1}^\infty\prod_{i=1}^j\tau_i=\sum_{j=1}^\infty\frac{1}{j+1}=\infty$$
and
$$\sum_{j=1}^\infty\frac{1}{j+1}\prod_{i=1}^j\tau_i=\sum_{j=1}^\infty\frac{1}{(j+1)^2}<\infty.$$
Thus none of the conditions from \cref{C:nmlinks} are satisfied.  However since $\tau_i<1$ for all $i$ we know from the theorem of Sher and Armentrout (\cref{theorem:armentrout-sher}) that $\mcD$ does not shrink.
\end{example}

\begin{example}
For our second example, let $L^{2s}$ be a $(2s^2,1)$-link and let $L^{2s+1}$ be a $(2,(s+1)^2)$-link. Then $\tau_{2s}=s^2, \tau_{2s+1}=\frac{1}{(s+1)^2}, n_{2s}=2s^2$ and $n_{2s+1}=1$. Therefore,
$$\sum_{j=1}^\infty\prod_{i=1}^j\tau_i=\sum_{j=1}^\infty\left(1+\tfrac{1}{j^2}\right)=\infty$$
and
$$\sum_{j=1}^\infty\tfrac{1}{n_j}\prod_{i=1}^j\tau_i=\sum_{j=1}^\infty\frac{3}{2j^2}<\infty.$$
So once again none of the conditions from \cref{C:nmlinks} are satisfied.  This time, by \cref{P:nmlink} we have $D_{2s}(k)=\lceil \tfrac{k}{s^2}\rceil-1$ and $D_{2s+1}(k)=(s+1)^2k-1$. Therefore, $$D_{2s+2}(D_{2s+1}(k))=\lceil k-\tfrac{1}{(s+1)^2}\rceil -1=k-1$$
and so by \cref{thmA} we see that $\mcD$ does shrink.
\end{example}

\bibliographystyle{annotate}
\bibliography{markbib1}

\begin{thebibliography}{Arm70}

\bibitem[AR65]{Andrews-Rubin-65}
J.~J. Andrews and L.~Rubin.
\newblock Some spaces whose product with {$E^{1}$} is {$E^{4}$}.
\newblock {\em Bull. Amer. Math. Soc.}, 71:675--677, 1965.


\bibitem[Arm66]{Armentrout-66}
S.~Armentrout.
\newblock Decompostions of {${E}^{3}$} with a compact {${\rm 0}$}-dimensional
  set of nondegenerate elements.
\newblock {\em Trans. Amer. Math. Soc.}, 123:165--177, 1966.


\bibitem[Arm70]{Armentrout-70}
S.~Armentrout.
\newblock On the strong local simple connectivity of the decomposition spaces
  of toroidal decompositions.
\newblock {\em Fund. Math.}, 69:15--37, 1970.


\bibitem[AS89]{Ancel-Starbird-1989}
F.~D. Ancel and M.~P. Starbird.
\newblock The shrinkability of {B}ing-{W}hitehead decompositions.
\newblock {\em Topology}, 28(3):291--304, 1989.


\bibitem[Bin52]{Bing-shrinking-bing-doubles}
R.~H. Bing.
\newblock A homeomorphism between the {$3$}-sphere and the sum of two solid
  horned spheres.
\newblock {\em Ann. of Math. (2)}, 56:354--362, 1952.


\bibitem[Bin57]{Bing-57-handlebody-decomps}
R.~H. Bing.
\newblock A decomposition of {$E^3$} into points and tame arcs such that the
  decomposition space is topologically different from {$E^3$}.
\newblock {\em Ann. of Math. (2)}, 65:484--500, 1957.


\bibitem[Bin62]{Bing-62-non-shrinking}
R.~H. Bing.
\newblock Point-like decompositions of {${E}^{3}$}.
\newblock {\em Fund. Math.}, 50:431--453, 1961/1962.


\bibitem[BZ85]{Burde-Zieschang:1985-1}
Gerhard Burde and Heiner Zieschang.
\newblock {\em Knots}.
\newblock Walter de Gruyter \& Co., Berlin, 1985.


\bibitem[Coc90]{Cochran-derivatives-memoir-90}
T.~D. Cochran.
\newblock Derivatives of links: {M}ilnor's concordance invariants and
  {M}assey's products.
\newblock {\em Mem. Amer. Math. Soc.}, 84(427):x+73, 1990.


\bibitem[Dav07]{Daverman-86}
R.~J. Daverman.
\newblock {\em Decompositions of manifolds}.
\newblock AMS Chelsea Publishing, Providence, RI, 2007.
\newblock Reprint of the 1986 original.


\bibitem[Edw80]{Edwards-ICM-80}
R.~D. Edwards.
\newblock The topology of manifolds and cell-like maps.
\newblock In {\em Proceedings of the {I}nternational {C}ongress of
  {M}athematicians ({H}elsinki, 1978)}, pages 111--127, Helsinki, 1980. Acad.
  Sci. Fennica.

\bibitem[Fer]{ferry-gt-notes}
S.~Ferry.
\newblock Notes on {G}eometric {T}opology.
\newblock Available at http://www.math.rutgers.edu/$\sim$sferry/ps/geotop.pdf.

\bibitem[FQ90]{FQ}
M.~H. Freedman and F.~Quinn.
\newblock {\em Topology of Four Manifolds}.
\newblock Princeton Univ. Press, 1990.


\bibitem[Fre82]{Freedman-JDG-82}
M.~H. Freedman.
\newblock The topology of four-dimensional manifolds.
\newblock {\em J. Differential Geom.}, 17(3):357--453, 1982.


\bibitem[Mil57]{Mil57}
J.~W. Milnor.
\newblock Isotopy of links.
\newblock {\em Alg. geom. and top. {A} symposium in honor of {S}. {L}efschetz},
  pages 280--306, 1957.


\bibitem[She67]{Sher-67}
R.~B. Sher.
\newblock Toroidal decompositions of {$E^{3}$}.
\newblock {\em Fund. Math.}, 61:225--241, 1967.


\bibitem[Wri89]{Wright-Bing-Whitehead-1989}
D.~G. Wright.
\newblock Bing-{W}hitehead {C}antor sets.
\newblock {\em Fund. Math.}, 132(2):105--116, 1989.


\end{thebibliography}
\end{document}